\documentclass{amsart}
\usepackage{amssymb}
\usepackage[enableskew]{youngtab}
\usepackage[boxsize=21pt]{ytableau}
\usepackage[colorlinks=true, pdfstartview=FitV, linkcolor=blue, citecolor=blue, urlcolor=blue]{hyperref}
\usepackage{mathrsfs}

\theoremstyle{plain}
\newtheorem{thm}{Theorem}[section]
\newtheorem{lemma}[thm]{Lemma}

\newtheorem{prop}[thm]{Proposition}
\newtheorem{cor}[thm]{Corollary}
\theoremstyle{definition}
\newtheorem{dfn}[thm]{Definition}
\newtheorem{ex}[thm]{Example}
\newtheorem{remark}[thm]{Remark}

\numberwithin{equation}{section}

\usepackage{amsaddr}
\usepackage{amsmath,amscd}

\newcommand{\arxiv}[1]{\href{https://arxiv.org/abs/#1}{\texttt{arXiv:#1}}}
\newcommand{\textcir}[1]{\raisebox{-.13pt}{\textcircled{\raisebox{-0.36pt} {#1}}}}

\begin{document}
\title[$P$-Schur positive $P$-Grothendieck polynomials]{$P$-Schur positive $P$-Grothendieck polynomials}

\author[G.~Hawkes]{Graham Hawkes}
\address[G.~Hawkes]{Max-Planck-Institut f\"ur Mathematik, Vivatsgasse 7, 53111 Bonn, Germany}
\email{hawkes@math.ucdavis.edu}
\thanks{G. Hawkes is supported by the Max-Planck-Institut f\"ur Mathematik.}
\begin{abstract}
The symmetric Grothendieck polynomials generalize Schur polynomials and are Schur-positive by degree. Combinatorially this is manifested as the generalization of semistandard Young tableaux by set-valued tableaux. We define a (weak) symmetric $P$-Grothendieck polynomial which generalizes $P$-Schur polynomials in the same way. Combinatorially this is manifested as the generalization of shifted semistandard Young tableaux by a new type of tableaux which we call shifted multiset tableaux.
\end{abstract}
\keywords{Combinatorics [05], Grothendieck polynomials, $P$-Schur polynomials, multiset tableaux, RSK correspondence}
\maketitle

\ytableausetup{boxsize=0.8cm}

\section{Introduction}

\subsection{Overview}

In this paper we will be discussing certain families of polynomials in a set of indeterminants $x_1,\ldots,x_n$.  Each of these families will be indexed by either partitions or else strict partitions (the subset of partitions with no two parts equal).  Moreover, the polynomials  in each family are symmetric under permutations of the indeterminants $x_1,\ldots,x_n$.  The most basic of these families will be the Schur polynomials. These polynomials appear in many areas of algebra and combinatorics and have a number of important properties: they form a basis for the space of symmetric polynomials in  $x_1,\ldots,x_n$, which is graded by degree, they are self-dual under the Hall inner product, they describe the characters of polynomial representations of $GL(n, \mathbb{C})$, and any product of Schur polynomials can be expressed as a nonnegative integral  linear combination of other Schur polynomials. 

$P$-Schur polynomials, like Schur polynomials, are homogeneous of degree equal to the size of the indexing partition.  Since they are only defined for strict partitions they cannot form a basis for the space of symmetric polynomials, although they do form a linearly independent set. Also, like  Schur polynomials they describe the characters of certain representations. 
 Further, any product of $P$-Schur polynomials can be written as a nonnegative integral linear combination of other $P$-Schur polynomials.  Finally, we note that $P$-Schur polynomials themselves can be written  as  nonnegative integral  linear combinations of Schur polynomials.

Symmetric Grothendieck polynomials and weak symmetric Grothendieck polynomials are $K$-theoretical generalizations of Schur polynomials.  In particular, the lowest degree part of a (weak) symmetric  Grothendieck polynomial associated to a certain partition returns the Schur polynomial indexed by the same partition.  Like Schur and $P$-Schur polynomials, they form important linearly independent sets in the ring of symmetric polynomials from algebraic as well as geometric perspectives.  Although symmetric Grothendieck polynomials and their weak versions share many properties (e.g., they both can be written as nonnegative integral linear combinations of Schur polynomials),  an essential difference between the two is that whereas the former are bounded in degree (by for instance $n |\lambda|$ where $|\lambda|$ means the size of the partition to which the polynomial is associated) the latter are not (and are actually formal power series rather than polynomials, strictly speaking).  

We will have occasion to study the weak version in this paper, and will define new polynomials that complete the analogy ``Schur polynomials  are to weak symmetric Grothendieck polynomials as $P$-Schur polynomials are to $\ldots$"   in a way that makes sense on as many levels as possible.  These polynomials will be ``defined" in two ways (formally, one may take either to be the definition and the other as a consequence).  The first definition will be a determinantal definition which is in analogy to the classical definition of Schur polynomials. It can be seen from this definition that our ``polynomial" is really a formal power series, because, for instance, its degree is not bounded. The second definition will be a combinatorial one, given as a generating function over certain tableaux.  These tableaux  will be of shifted shape, contain both primed and unprimed entries, and allow more than one entry per box.  Our interest in these polynomials and their potential importance will be discussed in the next two subsections.

\subsection{Motivation}

On the most basic level our goal is to complete the following diagram of important (see next subsection) polynomials in a way that makes as much algebraic and combinatorial sense as possible:\\

\begin{center}
$\begin{CD}
\hbox{\textbf{\small{Schur polynomials}}} @> \hbox{\em{\footnotesize{Type A to Type B}}}>> \hbox{\textbf{\small{$P$-Schur polynomials}}} \\
@VV\hbox{\em{\footnotesize{K-theory}}} V@VV \hbox{\em{\footnotesize{K-theory}}}V\\
\hbox{\textbf{\small{Grothendieck polynomials}}} @> \hbox{\em{\footnotesize{Type A to Type B}}} >> \hbox{\textbf{\small{TBD}}}
\end{CD}$
\end{center}

\bigskip

Here we are actually referring to the weak symmetric version of Grothendieck polynomials.  In order to define the missing polynomials we take algebra as our starting point rather than geometry or combinatorics because, as we will see, there is more or less a unique (up to certain dualities) solution to the problem considered from this point of view.  We then work out a combinatorial model that fits this choice.  This amounts to completing the following additional diagram:\\
\begin{center}
$\begin{CD}
\hbox{\textbf{\small{Semistandard Young tableaux}}} @> \hbox{\em{\footnotesize{straight to shifted}}}>> \hbox{\textbf{\small{Semistandard shifted tableaux}}} \\
@VV\hbox{\em{\footnotesize{element to set}}} V@VV \hbox{\em{\footnotesize{element to set}}}V\\
\hbox{\textbf{\small{Multiset valued tableaux}}} @> \hbox{\em{\footnotesize{straight to shifted}}} >> \hbox{\textbf{\small{TBD}}}
\end{CD}$
\end{center}

\subsection{Background and Goals}

Symmetric Grothendieck polynomials and weak symmetric Grothendieck polynomials are families of nonhomogeneous symmetric polynomials indexed  by partitions. They are related to both the more general Grothendieck polynomials of Lascoux and Sch\"utzenberger~\cite{LS82,LS83} as well as the stable Grothendieck polynomials of Fomin and Kirillov~\cite{FK94,FK96}.  These relationships are explicated in places such as ~\cite{MPS18} and ~\cite{HS20}. For these reasons symmetric and weak symmetric Grothendieck polynomials are fundamental building blocks in the subject of nonhomogeneous symmetric functions.

Moreover, they are natural generalizations of Schur polynomials to nonhomogeneous symmetric polynomials:  The Schur polynomial $s_{\mu}$ can be defined as the determinant of a certain matrix, $M=\{M_{ij}\}$, divided by the Vandermonde determinant.  If each entry $M_{ij}$ of this matrix is multiplied by $(1+x_i)^{\mu_j}$ this process instead gives the symmetric Grothendieck polynomial; if divided by $(1-x_i)^{\mu_j}$ it instead gives the weak symmetric Grothendieck polynomial.  On the combinatorial side, the definitions of symmetric Grothendieck polynomials and weak symmetric Grothendieck polynomials are obtained by generalizing the concept of semistandard Young tableaux to set-valued tableaux ~\cite{Buch02} or multiset-valued tableaux  ~\cite{LP07}, respectively.  Interestingly, both symmetric Grothendieck polynomials and weak symmetric Grothendieck polynomials are in turn themselves Schur positive by degree (see ~\cite{MPS18} and ~\cite{HS20} for crystal interpretations of these facts).

Since the $P$-Schur function $P_{\mu}$ also has a determinantal definition  it is a natural question to ask whether these phenomena occur if we try to generalize $P$-Schur functions in the same manner (i.e., multiplying entries of the relevant matrix by $(1+x_i)^{\mu_j}$ or dividing them by $(1-x_i)^{\mu_j})$.  In this paper, we give an in depth analysis of the second case.  We call the resulting polynomial $\mathfrak{P}_{\mu}(\mathbf{x})$.  The most interesting results are that, in analogy to the symmetric and weak symmetric Grothendieck case:

\begin{enumerate}
\item{$\mathfrak{P}_{\mu}(\mathbf{x})$ is $P$-Schur positive by degree.}
\item{There is a combinatorial definition for $\mathfrak{P}_{\mu}(\mathbf{x})$  generalizing that of $P_{\mu}(\mathbf{x})$.}
\end{enumerate}
 The combinatorial definition employs a new type of (rather unexpected) tableaux, which we call shifted multiset tableaux.  We give the $P$-Schur expansion of $\mathfrak{P}_{\mu}(\mathbf{x})$  explicitly using certain ``maximal" shifted multiset tableaux.  

The properties above hint that our polynomials are natural analogs to weak symmetric Grothendieck polynomials and therefore likely to be useful in the analogous settings to many of the settings in which the latter are useful. Moreover, the new tableaux themselves should be of combinatorial interest.  Indeed the fact that $\mathfrak{P}_{\mu}(\mathbf{x})$ is Schur $P$-positive implies that  the set of shifted multiset tableaux can be given a certain combinatorial structure known as a crystal.  The paper leaves as an open problem the task of giving a  realization of these crystals through explicit  combinatorial rules on the set of  shifted multiset tableaux.  Similar endeavors have been undertaken successfully for other types of tableaux, in particular we mention  \cite{Gal17},  \cite{HPS17},  \cite{MPS18}, \cite{Hir19}, and \cite{HS20}.

In the paper we study both the well-known weak symmetric Grothendieck polynomial, $\mathfrak{J}_{\mu}(\mathbf{x})$, as well as the new $\mathfrak{P}_{\mu}(\mathbf{x})$ in order to highlight the similarities. Although the analysis of  $\mathfrak{J}_{\mu}(\mathbf{x})$ is not entirely new, we also introduce a new multiparameter $\mathbf{t}=t_1,\ldots, t_{\ell}$ deformation $\mathfrak{J}_{\mu}(\mathbf{x},\mathbf{t})$\footnote{ It should also be mentioned that the $t$-deformation, $\mathfrak{J}_{\mu}(\mathbf{x},\mathbf{t})$, has combinatorial significance in terms of the crystal structure on multiset tableaux as given in \cite{HS20}.  In particular, the combinatorial statistics on multiset tableaux that determine the exponents of $\mathbf{t}$  in the combinatorial definition of $\mathfrak{J}_{\mu}(\mathbf{x},\mathbf{t})$  are constant on each connected component of the crystal of \cite{HS20}.  One would expect in turn that a natural crystal structure on shifted multiset tableaux ought to preserve the combinatorial statistics on shifted multiset tableaux that determine the exponents of $\mathbf{t}$  in the combinatorial definition of $\mathfrak{P}_{\mu}(\mathbf{x},\mathbf{t})$ given in this paper.  This should provide clues on how to properly construct said crystal. }.  Finally we give a similar deformation, $\mathfrak{P}_{\mu}(\mathbf{x},\mathbf{t})$, of $\mathfrak{P}_{\mu}(\mathbf{x})$.  These deformations are given because they are natural both from the algebraic and combinatorial definitions of the polynomials, and clarify the relationship between the two definitions.

We note that similar generalizations of $P$-Schur polynomials such as in ~\cite{IN13} and ~\cite{HKPWZZ17} have been made, but are distinct from  $\mathfrak{P}_{\mu}(\mathbf{x})$ (for example, they are not $P$-Schur positive and are not constructed by  generalizing the determinantal  definition of the $P$-Schur function).

\section{Lemma}

We begin with a basic lemma about how to multiply symmetric polynomials by a sequence of homogeneous symmetric polynomials in a weakly increasing number of variables.

Let $\mu=(\mu_1,\ldots,\mu_n)$ be a partition into distinct parts of some positive integer and fix integers $n \geq c_{\ell} \geq \cdots \geq c_1$.  Now consider $\mu$ as being represented by its Young diagram.  In what follows we will successively append boxes to the right side of this diagram in $\ell$ steps.  In what follows we consider a list of $\ell$ nonnegative integers, $\mathbf{T}=T_{1},\ldots,T_{\ell}$.   In step $h$, $T_h$ boxes will be appended and these boxes must be appended in rows at or above row $c_h$  (note that this process may result in the diagram ceasing to be a partition if a row becomes longer than the row above it, in this case we refer to the shape using the more general term ``composition").   More precisely, define a $\mathbf{T}$-extension of $\mu$ to be a sequence of diagrams, $\lambda=\lambda^{\ell} \supseteq \cdots \supseteq \lambda^1 \supseteq \lambda^0=\mu$ such that $|\lambda^h|-|\lambda^{h-1}|=T_h$ and $\lambda^h_k=\lambda^{h-1}_k$ for $k>c_h$ for all $h \in [1,\ell]$. A $\mathbf{T}$-extension of $\mu$ is called \emph{good} if $\lambda^h_k < \lambda^{h-1}_{k-1}$ for $k \in [2,c_h]$ for all $h \in [1,\ell]$. A $\mathbf{T}$-extension which is not \emph{good} is called \emph{bad}. In particular, every diagram in a \emph{good} $\mathbf{T}$-extension is a partition.

\begin{ex}

Let $\mu=(4,3,2,1)$ with $(c_1,c_2,c_3,c_4)=(2,3,3,4)$  and $\mathbf{T}=(4,6,4,7)$.  
\begin{eqnarray*}
\ytableausetup{notabloids}
\begin{ytableau}
 \phantom{1} &  &  & & *(green) \phantom{1} & *(green)  \phantom{1} &*(yellow)  \phantom{1} &*(yellow)  \phantom{1} &*(red)  \phantom{1} \\
 \phantom{1}&  & & *(green)  \phantom{1} & *(green)  \phantom{1} &*(orange)  \phantom{1}  &*(orange)  \phantom{1}  &*(orange)  \phantom{1} &*(red)  \phantom{1} \\
   \phantom{1} & & *(yellow)  \phantom{1}  & *(yellow)  \phantom{1} & *(yellow)  \phantom{1}  & *(yellow)  \phantom{1} & *(orange)  \phantom{1} &*(red)  \phantom{1}  \\
 \phantom{1} &  *(red)  \phantom{1} &*(red)  \phantom{1} & *(red)  \phantom{1} & *(red)  \phantom{1} \\
\end{ytableau}
\end{eqnarray*}
Boxes are added in the  steps below to form  the $\mathbf{T}$-extension shown above. 
\begin{enumerate}
\item $T_1=4$ (green) boxes  added on or above row $c_1=2$ to form $\lambda^1=(6,5,2,1)$.
\item  $T_2=6$ (yellow) boxes  added on or above row $c_2=3$ to form $\lambda^2=(8,5,6,1)$.
\item $T_3=4$ (orange) boxes  added on or above row $c_3=3$ to form $\lambda^3=(8,8,7,1)$.
\item  $T_4=7$ (red) boxes  added on or above row $c_4=4$ to form $\lambda^4=(9,9,8,5)$.
\end{enumerate}
This $\mathbf{T}$-extension is bad because, for instance, we have $\lambda_3^3=7 \geq 5 = \lambda_2^2$.
\end{ex}

\begin{lemma}\label{hmult}
\begin{eqnarray*}
\sum_{\sigma \in S_n} { \mathbf{sgn}  (\sigma)}h_{T_{\ell}}(x_{\sigma_1},\ldots, x_{\sigma_{c_{\ell}}}) \cdots h_{T_1}(x_{\sigma_1},\ldots, x_{\sigma_{c_1}})x_{\sigma_1}^{\mu_1} \cdots x_{\sigma_n}^{\mu_n}\\
=\sum_{\lambda=\lambda^{\ell} \supseteq \cdots \supseteq \lambda^1 \supseteq \lambda^0=\mu}
\left(\sum_{\sigma \in S_n}  \mathbf{sgn}  (\sigma)x_{\sigma_1}^{\lambda_1} \cdots x_{\sigma_n}^{\lambda_n}\right)
\end{eqnarray*}
where the sum is over all good $\mathbf{T}$-extensions.
\end{lemma}

\begin{proof}
It suffices to show that
\begin{eqnarray*}
\sum_{\lambda=\lambda^{\ell} \supseteq \cdots \supseteq \lambda^1 \supseteq \lambda^0=\mu}
\left(\sum_{\sigma \in S_n}  \mathbf{sgn}  (\sigma)x_{\sigma_1}^{\lambda_1} \cdots x_{\sigma_n}^{\lambda_n}\right)=0
\end{eqnarray*}
where the sum is over all bad $\mathbf{T}$-extensions.  Indeed, consider the summand corresponding to a fixed $\sigma$  in the top line of the equation of Lemma \ref{hmult}. When the terms of this single summand are multiplied out, each resulting monomial  specifies a $\mathbf{T}$-extension of $\mu$ by considering this monomial as being formed by starting with $x_{\sigma_1}^{\mu_1} \cdots x_{\sigma_n}^{\mu_n}$ and then multiplying by some monomial appearing in $h_{T_1}(x_{\sigma_1},\ldots, x_{\sigma_{c_1}})$ and then multiplying by some monomial appearing in $h_{T_2}(x_{\sigma_1},\ldots, x_{\sigma_{c_2}})$, et cetera.  Moreover, each $\mathbf{T}$-extension of $\mu$ is formed in this way exactly once.   Since this is simultaneously true for all $\sigma \in S_n$, Lemma \ref{hmult} is true if we take the outer sum in the second line over \emph{all} $\mathbf{T}$-extensions.  If we verify the equation at the beginning of this paragraph, then we can replace \emph{all} with \emph{all good}.

 It suffices to find a sign-changing involution, $\iota$, on the set of pairs of the form $(\sigma, \Lambda)$ where $\sigma \in S_n$, $\Lambda$ is a bad $\mathbf{T}$-extension and the sign of the pair is the sign of the permutation $\sigma$, such that $\iota$ has the following property: If $\iota(\sigma,\Lambda)=(\bar{\vphantom{<}\sigma},\bar{\Lambda})$ where $\lambda$ is the largest composition of $\Lambda$ and $\bar{\lambda}$ is the largest composition of $\bar{\Lambda}$ then $\bar{\lambda}_{(\bar{\vphantom{<}\sigma}^{-1}(p))}=\lambda_{(\sigma^{-1}(p))}$ for all $p \in [1,n]$.

Define $\iota(\sigma,\Lambda)$ as follows: Suppose $\Lambda$ is the bad $\mathbf{T}$-extension $\lambda=\lambda^{\ell} \supseteq \cdots \supseteq \lambda^1 \supseteq \lambda^0=\mu$. Choose $i$ minimal such that there exists some $k \in [2,c_i]$ such that $\lambda^i_k \geq \lambda^{i-1}_{k-1}$. Choose the minimal such $k$, and then choose the minimal $j \in [1,k)$ such that $\lambda^i_k \geq \lambda^{i-1}_{j}$. Define $\bar{\vphantom{<}\sigma}(m)=\sigma(m)$ for $m \notin \{j,k\}$, $\bar{\vphantom{<}\sigma}(j)=\sigma(k)$, and $\bar{\vphantom{<}\sigma}(k)=\sigma(j)$. Next, for $h<i$ define $\bar{\lambda}^h=\lambda^h$. For $h \geq i$ define $\bar{\lambda}^h_m=\lambda^h_m$ for $m \notin \{j,k\}$, $\bar{\lambda}^h_j=\lambda^h_k$, and $\bar{\lambda}^h_k=\lambda^h_j$. Set $\iota(\sigma,\Lambda)=(\bar{\vphantom{<}\sigma},\bar{\Lambda})$ where $\bar{\Lambda}$ is the bad $\mathbf{T}$-extension $\bar{\lambda}=\bar{\lambda}^{\ell} \supseteq \cdots \supseteq \bar{\lambda}^1 \supseteq \bar{\lambda}^0=\mu$. (That the $\supseteq$ are correct, and that $\bar{\Lambda}$ is a bad $\mathbf{T}$-extension is proved below).

Note the following properties of $\iota$.

\begin{enumerate}
\item{$\iota(\sigma,\Lambda)$ has the opposite sign to $(\sigma,\Lambda)$.}
\item{$\bar{\lambda}_{(\bar{\vphantom{<}\sigma}^{-1}(p))}=\lambda_{(\sigma^{-1}(p))}$ for all $p \in [1,n]$.}
\item{$\bar{\Lambda}$ is a $\mathbf{T}$-extension.
\begin{itemize}
\item{That $|\bar{\lambda}^h|-|\bar{\lambda}^{h-1}|=T_h$ is immediate.}
\item{Suppose that $m>c_h$, we wish to check that $\bar{\lambda}^h_m=\bar{\lambda}^{h-1}_m$. Now if $m \in \{j,k\}$ and $h \geq i$ we have $j,k \leq c_i \leq c_h$ so the condition $m>c_h$ is impossible to attain.
Thus we may assume that $m \notin \{j,k\}$ or $h < i$ in which case we have $\lambda^h_m=\bar{\lambda}^h_m$ and $\lambda^{h-1}_m=\bar{\lambda}^{h-1}_m$ so that  $\lambda^h_m=\lambda^{h-1}_m$ implies  $\bar{\lambda}^h_m=\bar{\lambda}^{h-1}_m$. }
\item{Next, it is clear that $\bar{\lambda}^h \supseteq \bar{\lambda}^{h-1}$ if $h \neq i$ and also that $\bar{\lambda}^h_m>\bar{\lambda}^{h-1}_m$ for $m \notin \{j,k\}$. We need only check that $\bar{\lambda}^i_j \geq \bar{\lambda}^{i-1}_j$ and $\bar{\lambda}^i_k \geq \bar{\lambda}^{i-1}_k$. The first is equivalent to saying that $\lambda^i_k \geq \lambda^{i-1}_j$ which is true by the choice of $j$ and $k$. The second is equivalent to saying that $\lambda^i_j \geq \lambda^{i-1}_k$ but $\lambda^i_j \geq \lambda^{i-1}_j$ since $\lambda^i \supseteq \lambda^{i-1}$ and $\lambda^{i-1}_j \geq \lambda^{i-1}_k$ by minimality of $i$.}
\end{itemize}}
\item{$\bar{\Lambda}$ is a bad $\mathbf{T}$-extension. Indeed, $\bar{\lambda}^i_k=\lambda^i_j \geq \lambda^{i-1}_j=\bar{\lambda}^{i-1}_j$.}
\item{ $\iota^2(\sigma,\Lambda)=(\sigma,\Lambda)$. It is clear from the definitions that this is true as long as the values of $i,k,j$ chosen when applying $\iota$ to $(\sigma, \Lambda)$ are the same as those (say $\bar{i},\bar{k},\bar{j}$) chosen when applying $\iota$ to $(\bar{\vphantom{<}\sigma},\bar{\Lambda})$. Clearly $\bar{i} \geq i$, and, by the step above, $\bar{i} \leq i$, so $\bar{i}=i$. If $j \neq \bar{k} < k$ then $\lambda^i_{\bar{k}}=\bar{\lambda}^i_{\bar{k}} \geq \bar{\lambda}^{i-1}_{\bar{k}-1}=\lambda^{i-1}_{\bar{k}-1}$, contradicting the minimality of $k$. If $\bar{k}=j$ then $\lambda^i_k=\bar{\lambda}^i_j\geq \bar{\lambda}^{i-1}_{j-1}=\lambda^{i-1}_{j-1}$, contradicting the minimality of $j$. Since the step above implies $\bar{k} \leq k$ this means $\bar{k}=k$. Finally, if $\bar{j} < j$ then $\lambda^i_j=\bar{\lambda}^i_k \geq \bar{\lambda}^{i-1}_{\bar{j}}=\lambda^{i-1}_{\bar{j}}$, contradicting the minimality of $k$. Again, the step above means $\bar{j} \leq j$ so together we get $\bar{j}=j$.}
\end{enumerate}

This shows that $\iota$ is a well defined sign-changing involution with the desired property, proving the lemma.
\end{proof}


\section{$\mathfrak{J}_{\mu}(\mathbf{x},\mathbf{t})$ and multiset tableaux}

\subsection{Algebraic Definition of $\mathfrak{J}_{\mu}(\mathbf{x},\mathbf{t})$}

\emph{We will always work in $n$ variables and will set $V= \prod\limits_{i<j} (x_i-x_j)$. In general we will define a symmetric polynomial $f$ by defining the value of the skew-symmetric polynomial $V*f$ (where $*$ simply denotes multiplication).}

For a partition $\mu$ of $n$ parts (some may equal $0$), the weak symmetric Grothendieck polynomial in $n$ variables is defined by:

\begin{eqnarray*}
V*\mathfrak{J}_{\mu}(x_1,\ldots,x_n)= \sum\limits_{\sigma \in S_n}  \mathbf{sgn}  (\sigma) \prod\limits_{i} \left( \left(\frac{x_{\sigma_i}}{1-x_{\sigma_i}}\right)^{\mu_i}x_{\sigma_i}^{n-i} \right) 
\end{eqnarray*}

This formula can be obtained by applying the standard involution on symmetric functions to the determinantal definition for Grothendieck polynomials found in  \cite{IN13} among other places.
Next, we define a slight generalization of this polynomial. Suppose $\mu$ has longest part $\ell=\mu_1$. Let the weak symmetric Grothendieck polynomial in $n+\ell$ variables be defined by:

\begin{eqnarray*}
V*\mathfrak{J}_{\mu}(x_1,\ldots,x_n,t_1,\ldots,t_{\ell})= \sum\limits_{\sigma \in S_n}  \mathbf{sgn}  (\sigma) \prod\limits_{i} \left(\left(\frac{x_{\sigma_i}}{1-t_{\ell}x_{\sigma_i}}\right)\cdots \left(\frac{x_{\sigma_i}}{1-t_{\ell-\mu_i+1}x_{\sigma_i}}\right) x_{\sigma_i}^{n-i}\right)
\end{eqnarray*}

Clearly $\mathfrak{J}_{\mu}(x_1,\ldots,x_n,1,\ldots,1)=\mathfrak{J}_{\mu}(x_1,\ldots,x_n)$ whereas $\mathfrak{J}_{\mu}(x_1,\ldots,x_n,0,\ldots,0)=s_{\mu}(x_1,\ldots,x_n)$ (see the classical definition of the Schur function as found in 7.15 of \cite{STA99}). Note that the coefficient of $ t_1^{T_1} \cdots t_{\ell}^{T_{\ell}}$ in $V*\mathfrak{J}_{\mu}(x_1,\ldots,x_n,t_1,\ldots,t_{\ell})$ is given by:

\begin{eqnarray*}
\sum_{\sigma \in S_n}  \mathbf{sgn}  (\sigma)h_{T_{\ell}}(x_{\sigma_1},\ldots, x_{\sigma_{c_{\ell}}}) \cdots h_{T_1}(x_{\sigma_1},\ldots, x_{\sigma_{c_1}})x_{\sigma_1}^{\mu_1+n-1} \cdots x_{\sigma_n}^{\mu_n+0}\\
\end{eqnarray*}
where $(c_{\ell},\ldots,c_1)=\mu'$. Since $n \geq c_{\ell} \geq \cdots \geq c_1$ Lemma \ref{hmult} implies that this coefficient is:

\begin{eqnarray*}
\sum_{\lambda=\lambda^{\ell} \supseteq \cdots \supseteq \lambda^1 \supseteq \lambda^0=\mu+\delta}
\left(\sum_{\sigma \in S_n}  \mathbf{sgn}  (\sigma)x_{\sigma_1}^{\lambda_1} \cdots x_{\sigma_n}^{\lambda_n}\right)
\end{eqnarray*}
where the outer sum is over all good $\mathbf{T}$-extensions and $\delta=(n-1,\ldots,0)$.
Fix some $\lambda \supseteq \mu + \delta$ and let $\rho=\lambda-\delta$ and consider the set of semistandard Young tableaux of shape $\rho/\mu$ and weight $T_1,\ldots,T_{\ell}$ such that every entry $i$ occurs on or above row $c_i$.  We claim that each good $\mathbf{T}$-extension $\lambda=\lambda^{\ell} \supseteq \cdots \supseteq \lambda^1 \supseteq \lambda^0=\mu+\delta$ encodes such a tableau in a bijective fashion.  Indeed, the bijection is given by first considering the tableau of shape $\lambda/(\mu+\delta)$ where each strip $\lambda^i/\lambda^{i-1}$ is filled with $i$s  for each $i \in [1,\ell]$, and then moving all entries in row $j$ to the left by $n-j$ positions for each $j \in [1,n]$ (and eliminating excess empty boxes).  

 Since the inner sum in the expression above is precisely the classical definition of $V*s_{\rho}(x_1,\ldots,x_n)$ where $s_{\rho}(x_1,\ldots,x_n)$ is the Schur function, it now follows that the expression above is the same as:
\begin{eqnarray*}
V*\sum_{\rho \supseteq \mu}
(M^{\mathbf{T}}_{\rho/\mu}) s_{\rho}(x_1,\ldots,x_n)
\end{eqnarray*}
where $M^{\mathbf{T}}_{\rho/\mu}$ is the number of semistandard Young tableaux of shape $\rho/\mu$ and weight $T_1,\ldots,T_{\ell}$ such that every entry $i$ occurs on or above row $c_i$.

\begin{dfn}
Let $\mu$ be a partition with $n$ parts and conjugate $\mu'=(c_{\ell},\ldots,c_1)$. We define a \textbf{\emph{restricted tableau}} of shape $\lambda/\mu$, or element of $RT(\lambda/\mu)$, to be a semistandard Young tableau of shape $\lambda/\mu$ in the alphabet $\{1,\ldots,\ell\}$ such that each entry $i$ occurs on or above row $c_i$. If $R \in RT(\lambda/\mu)$ then the weight of $R$, denoted $wt(R)$ is the vector $(w_1,\ldots,w_{\ell})$ where $w_i$ is the number of $i$s that appear in $R$.
\end{dfn}

\begin{ex}\label{resex}
Let $\lambda=(7,6,5,4)$ and $\mu =(4,3,3,2)$ so that $c_4=4$, $c_3=4$, $c_2=3$, $c_1=1$.

\begin{eqnarray*}
\begin{ytableau}
*(green) \cdot & *(green) \cdot &*(green) \cdot &*(green) \cdot & *(green) 1 & *(green) 2 & *(green) 3 \\
*(yellow) \cdot & *(yellow) \cdot &*(yellow) \cdot & *(yellow) 2 & *(yellow) 2 & *(yellow) 4 \\
*(yellow) \cdot & *(yellow) \cdot &*(yellow) \cdot & *(yellow) 3 & *(yellow) 3 \\
*(red) \cdot &*(red) \cdot & *(red) 3 & *(red) 4 \\
\end{ytableau}
\end{eqnarray*}
Since all 1s lie in the green, all 2s lie in the green or yellow, and all 3s and all 4s lie in the red, yellow, or green, this is an element of $RT(\lambda/\mu)$. It has weight $(1,3,4,2)$.

\end{ex}
With this definition, the computation before the definition shows:
\begin{thm}\label{spex}

Let $\mu$ be a partition with $n$ parts (some may equal $0$) and longest part equal to $\ell$ and set $\mathbf{t}=(t_1,\ldots , t_{\ell})$, $\mathbf{x}=(x_1,\ldots,x_n)$ then

\begin{eqnarray*}
\mathfrak{J}_{\mu}(x_1,\ldots,x_n,t_1,\ldots,t_{\ell})=
\sum_{\lambda \supseteq \mu} \sum_{R \in RT(\lambda/\mu)}
\mathbf{t}^{wt(R)}s_{\lambda}(\mathbf{x})
\end{eqnarray*}
\end{thm}

\subsection{Straight-shape multiset tableaux}

\begin{dfn}[\cite{LP07}]
Given a partition $\mu$, with conjugate $(c_{\ell},\ldots,c_1)=\mu'$ a \textbf{\emph{ multiset tableau}} of shape $\mu$, or an element of $MT(\mu)$ is a collection of boxes with $\mu_i$ boxes in each row and the rows left-justified, along with a filling of said boxes with the following properties.
\begin{enumerate}
\item{Each box contains a nonempty multiset of the numbers $\{1,2,\ldots\}$.}
\item{The maximum value of each box is strictly less than the minimum value of the box below it (if it exists) and weakly less than the minimum value of the box to its right (if it exists).}
\end{enumerate}
\end{dfn}

The weight, denoted $wt$, of a multiset tableau is the vector $(w_1,w_2, \ldots)$ where $w_i$ is the total number of $i$s appearing in the tableau. We label the columns from left to right by $\ell, \ell-1, \ldots, 1$. That is, by box $b_{ij}$ we refer to the box that is in the $i^{th}$ row from the top row and the $\ell-j+1^{st}$ column from the leftmost column. Define the column weight of a multiset tableau, $cw$, to be the vector $(T_1,\ldots,T_{\ell})$ where $T_i$ is the difference between the number of entries in column $i$ and the height of that column ($c_i$). By $|b_{ij}|$ we simply mean the total number of entries in box $b_{ij}$ and $|b_{ij}(x)|$ refers more specifically to the number of entries in box $b_{ij}$ in tableau $x$. By the nonemptiness property $|b_{ij}| \geq 1$ if box $b_{ij}$ exists and, by convention, is $0$ otherwise.
\begin{ex}
Let $\mu=(3,3,2)$. Then
\begin{eqnarray*}
\begin{ytableau}
11&12&333\\2&3&445\\34&4
\end{ytableau}
\end{eqnarray*}
is an element $P \in MT(\mu)$ with $wt(T)=(3,2,5,4,1)$ and $cw(P)=(4,1,2)$.
\end{ex}

\begin{dfn}
A \textbf{\emph{maximal multiset tableau}} of shape $\mu$, or element of $\overline{MT}(\mu)$, is a multiset tableau of shape $\mu$ with the following properties:
\begin{enumerate}
\item{Each box $b_{ij}$ may only contain $i$s.}
\item{For each $ i \geq 1$ and $ k \geq 1$ we have $\sum\limits_{1 \leq j \leq k} |b_{(i+1)j}|-|b_{i(j-1)}| \leq 1$}
\end{enumerate}
where by convention $|b_{i0}|=0$.  Note that  (2) when $k=\mu_1$ implies that each row $i$ contains at least as many entries as row $i+1$ since $|b_{(i)\mu_1}|\geq1$.  Along with (1) this implies the weight of a maximal multiset tableau is necessarily a partition.  
\end{dfn}

\begin{ex} \label{maxex}
Let $\mu=(4,3,3,1)$. Then
\begin{eqnarray*}
\begin{ytableau}
1&11&11&11\\22&2&222\\3&333&3\\44&44
\end{ytableau}
\end{eqnarray*}
is an element $P \in \overline{MT}(\mu)$ with $wt(T)=(7,6,5,4)$ and $cw(P)=(1,3,4,2)$.
\end{ex}

\begin{prop}
There is a bijection from the subset of $\overline{MT}(\mu)$ with weight $\lambda$ and column weight $\mathbf{T}$ to the subset of $RT(\lambda/\mu)$ with weight $\mathbf{T}$.
\end{prop}

\begin{proof}
Let $X$ be the subset of elements of $MT(\mu)$ with weight $\lambda$ and column weight $\mathbf{T}$ that satisfy property (1) above. Let $Y$ be the set of fillings (one entry per box) of shape $\lambda/\mu$ and weight $\mathbf{T}$ that are weakly increasing along rows such that every entry $i$ occurs on or above row $c_i$ (equivalently, row $i$ only contains entries greater than $\ell-\mu_i$). The map $x \rightarrow y$ where $y$ is defined by the property that for each $(i,j)$, row $i$ of $y$ contains exactly $|b_{ij}(x)|-1$ copies of $j$ is a bijection from $X$ to $Y$. Moreover if $x \rightarrow y$ then $x$ satisfies property (2) above if and only if the columns of $y$ are strictly increasing down rows: Indeed, if there is some $i$ and some $k$ such that $\sum\limits_{1 \leq j \leq k} |b_{(i+1)j}|-|b_{i(j-1)}| >1$ then for the minimal such $k$, the rightmost $k$ appearing in row $i+1$ of $y$ will lie in column $k+1-\sum\limits_{1 \leq j \leq k} |b_{(i+1)j}|$ whereas the rightmost $k-1$ (or rightmost instance of the greatest number less than $k-1$ if this row contains no $k-1$) appearing in row $i$ of $y$ will lie in column $k-\sum\limits_{1 \leq j \leq k} |b_{i(j-1)}|$, that is, strictly to the left of the aforementioned $k$. Hence said $k$ must lie below a number greater than or equal to $k$.  On the other hand, if in some column, say $p$, row $i+1$ of $y$ contains a $k$ that lies below some $k'$ in row $i$ with $k'\geq k$ then $k+1-\sum\limits_{1 \leq j \leq k} |b_{(i+1)j}| \leq p$  and $k-\sum\limits_{1 \leq j \leq k} |b_{i(j-1)}| > p$ so that we have $\sum\limits_{1 \leq j \leq k} |b_{(i+1)j}|-|b_{i(j-1)}| >1$. Since the elements of $Y$ that are strictly increasing down columns are exactly the elements of $RT(\lambda/\mu)$ with weight $\mathbf{T}$, the map restricted to the elements of $X$ that satisfy property (2) gives the desired bijection.
\end{proof}

\begin{ex}
The tableaux of examples \ref{resex} and \ref{maxex} correspond under this bijection.
\end{ex}

\begin{cor}\label{hwns}
Let $\mu$ be a partition with $n$ parts (some may equal $0$) and longest part equal to $\ell$ and set $\mathbf{t}=(t_1,\ldots , t_{\ell})$, $\mathbf{x}=(x_1,\ldots,x_n)$ then
\begin{eqnarray*}
\mathfrak{J}_{\mu}(x_1,\ldots,x_n,t_1,\ldots,t_{\ell})=
\sum_{P \in \overline{MT}_{\mu}}
\mathbf{t}^{cw(P)}s_{wt(P)}(\mathbf{x})
\end{eqnarray*}
\end{cor}

\subsection{Combinatorial Definition of $\mathfrak{J}_{\mu}(\mathbf{x},\mathbf{t})$}

In this section we will give an equivalent combinatorial definition of $\mathfrak{J}_{\mu}$. We will need to use the \textit{dual RSK} column insertion algorithm (see, for instance ~\cite{STA99}). We refer to dual RSK insertion of an element into a column, and the reverse insertion of an element under dual RSK as \textit{insert} and \textit{reverse insert}. These maps are reviewed below.

Let $K$ be a valid column (each box of $K$ contains exactly one number and the numbers strictly increase from top to bottom). One \textit{inserts} $a$ into $K$, denoted $a \rightarrow K$ as follows:
Let $\hat{a}$ denote the uppermost entry in $K$ such that $a \leq \hat{a}$. If $\hat{a}$ exists, replace $\hat{a}$ with $a$ and bump out $\hat{a}$. Otherwise, append $a$ to the bottom of $K$. The result is recorded as the pair $(K',\hat{a})$ if the second of this pair exists and just $K'$ otherwise.
On the other hand if $z \geq a$ for some $a \in K$ then we define \textit{reverse insertion} of $z$ into $K$ or $K \leftarrow z$ as follows: Let $\hat{z}$ denote the bottommost entry in $K$ such that $z \geq \hat{z}$. Replace $\hat{z}$ with $z$ and bump out $\hat{z}$. The result is recorded as the pair $(\hat{z},K')$.

Notice the basic properties:

\begin{enumerate}
\item{If $a \rightarrow K=K'$ then $K'$ is a valid column. }
\item{if $a \rightarrow K= (K',\hat{a})$ then $K'$ is a valid column.}
\item{If $K \leftarrow z=(\hat{z},K')$ then $K'$ is a valid column.}
\item{If $a \leq z$ then either
\begin{itemize}
\item{$z \rightarrow K=K'$ and $a \rightarrow K' =({K}^{\prime\prime},\hat{a})$ for some $\hat{a}$.}
\item{$z \rightarrow K=(K',\hat{z})$ and $a \rightarrow K' =({K}^{\prime\prime},\hat{a})$ where $\hat{a} \leq \hat{z}$.}
\end{itemize}}
\item{If $a \leq z$ and
$K \leftarrow a=(\hat{a},K')$ and $K' \leftarrow z=(\hat{z},{K}^{\prime\prime})$ then $\hat{a} \leq \hat{z}$.}
\end{enumerate}

\begin{prop} \label{bjs}
There is a bijection $\Psi: MT(\mu) \rightarrow \bigcup\limits_{\lambda \supseteq \mu} SSYT(\lambda) \times RT(\lambda/\mu)$, such that if $P \rightarrow (Q,R)$ then:
\begin{enumerate}
\item{$wt(P)=wt(Q)$.}
\item{$cw(P)=wt(R)$.}
\end{enumerate}

\end{prop}

\begin{proof}[Proof preliminaries]\renewcommand{\qedsymbol}{}
Fix $\mu$ a partition with conjugate $\mu'=(c_{\ell},\ldots,c_1)$. In the following, we will label the columns of \emph{any} tableau from left to right by $\ell, \ell-1,\ldots,1,0,-1,\ldots$.

 For $k \leq \ell$, define the set $MT_{k}(\lambda)$ to be the subset of elements of $MT(\lambda)$ that have only single entries in the boxes in columns $k-1,\ldots,1,0,-1,\ldots$. Define the set $RT_{k}(\lambda/\mu)$ to be the subset of elements of  $RT(\lambda/\mu)$ that have only entries from $\{1,2,\ldots,k-1\}$. Given a pair $(Q,R) \in MT_k(\lambda) \times RT_k(\lambda/\mu)$ define the weight and column weight of this pair as $wt(Q,R)=wt(Q)$ and $cw(Q,R)=cw(Q)+wt(R)$. To prove the proposition it suffices to find a weight and column weight preserving bijection for each $k$ from $\bigcup\limits_{\lambda \supseteq \mu} MT_k(\lambda) \times RT_k(\lambda/\mu)$ to $\bigcup\limits_{\lambda \supseteq \mu} MT_{k+1}(\lambda) \times RT_{k+1}(\lambda/\mu)$  (and then compose: $\Psi(P)=\Psi_{\ell}\circ \cdots \circ \Psi_1(P,Q_0)$ where $Q_0$ is the empty tableau of shape $\mu/\mu$). To do the former, it is enough to find a weight preserving bijection $\psi_k: MT_k(\lambda) \rightarrow \bigcup\limits_{\nu \supseteq \lambda} MT_{k+1}(\nu)$ where the union is over all $\nu$ such that $\nu/\lambda$ is a horizontal strip with no box below row $c_{k}$.  We then set $\Psi_k(Q,R)=(\psi_k(Q),R')$ where $R'$ is obtained by appending boxes to $R$ until its outer shape is the same as $\psi_k(Q)$  and filling each appended box with the entry $k$.

Before we can define $\psi_k$ we need to introduce the following map: Let $T \in MT_{k}(\lambda)$. Define $\textit{out}(T)$
as follows: First, in each box of column $k$ circle (one of) the minimum entry(s) from that box. Now find (one of) the largest noncircled entry(s) in column $k$ and remove it and \textit{insert} it into the column to the right of the column from which it was removed. After this, each time an element is bumped, \textit{insert} it into the next column to the right until some entry is eventually appended to a (possibly empty) column.  Note the following properties of $\textit{out}$.

\begin{enumerate}
\item{The path of positions where an element is bumped/appended moves weakly up as we move to the right.}
\item{The result of $\textit{out}$ is a multiset tableau.}
\item{If $\textit{out}(T)$ and $\textit{out}(\textit{out}(T))$ are both defined then the box that $\textit{out}$ appends to $\textit{out}(T)$ lies strictly to the right of the box that $\textit{out}$ appends to $T$.}
\end{enumerate}

\begin{ex} \label{outex}
Suppose that column $k$ is the second column from the left. Each $\longrightarrow$ represents an application of $\textit{out}$.
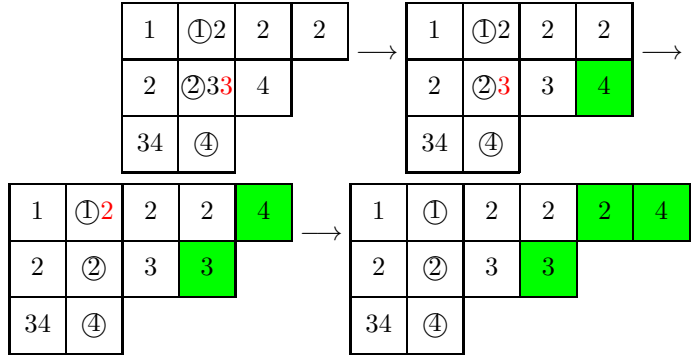
\begin{eqnarray*}
\begin{ytableau}
1&{\Large{\textcir{\normalsize{1}}}}1&2&2\\2& {\Large{\textcir{\normalsize{2}}}}3 \textcolor{red}{3}&4\\34&{\Large{\textcir{\normalsize{4}}}}\\
\end{ytableau}\longrightarrow
\begin{ytableau}
1&{\Large{\textcir{\normalsize{1}}}}1&2&2\\2& {\Large{\textcir{\normalsize{2}}}}\textcolor{red}{3}&3&*(green)4\\34&{\Large{\textcir{\normalsize{4}}}}\\
\end{ytableau} \longrightarrow\\
\begin{ytableau}
1&{\Large{\textcir{\normalsize{1}}}}\textcolor{red}{1}&2&2&*(green)4\\2& {\Large{\textcir{\normalsize{2}}}}&3&*(green)3\\34&{\Large{\textcir{\normalsize{4}}}}\\
\end{ytableau} \longrightarrow
\begin{ytableau}
1&{\Large{\textcir{\normalsize{1}}}}&1&2&*(green)2&*(green)4\\2& {\Large{\textcir{\normalsize{2}}}}&3&*(green)3\\34&{\Large{\textcir{\normalsize{4}}}}\\
\end{ytableau}
\end{eqnarray*}
Uncircled numbers being removed are shown in red, and the boxes being added appear in green.
\end{ex}

We will also need a map called $\textit{in}_b$. Let $T \in MT_{k}(\nu)$ for some $\nu$ such that $\nu/\lambda$ is a horizontal strip with no box below row $c_{k}$ and suppose $b$ is some corner box of this strip. First, in each box of column $k$ circle (one of) the minimum entry(s) from that box. Define $\textit{in}_b(T)$ as follows: Remove the entry from box $b$ and \textit{reverse insert} it into the column to the left. After this, each time an element is bumped \textit{reverse insert} it into the column to the left until an element is removed from column $k-1$. Then add this element to the lowest box of column $k$ such that the resulting column satisfies the column strict requirement in (2) of the definition of multiset tableau. Note the following properties of $\textbf{in}_b$.

\begin{enumerate}
\item{The path of positions where an element is bumped/added moves weakly down as we move to the left.}
\item{The result of $\textit{in}_b$ is a multiset tableau.}
\item{If $b'$ lies to the left of $b$ and if $\textit{in}_b(T)$ and $\textit{in}_{b'}(\textit{in}_b(T))$ are both defined then the element that $\textit{in}_{b'}$ adds to column $k$ of $\textit{in}_b(T)$ is greater than or equal to the element $\textit{in}_b$ adds to column $k$ of $T$. }
\end{enumerate}

Moreover, $\textit{out}$ and $\textit{in}_b$ are related as follows:

\begin{enumerate}
\item{If $\textit{out}$ appends box $b$ when applied to $T$, then $\textit{in}_b(\textit{out}(T))=T$.}
\item{If the element that $\textit{in}_b$ adds to column $k$ when applied to $T$ is the largest or tied for the largest uncircled element on column $k$ then $\textit{out}(\textit{in}_b(T))=T$ .}
\end{enumerate}

\begin{ex} \label{inex}
Let $k=2$. Then $\textit{in}_{red}(\textit{in}_{yellow}(\textit{in}_{green}(T)))=T'$ where:
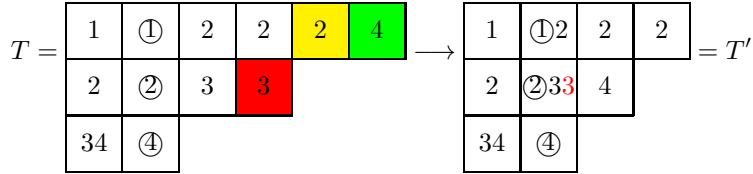
\begin{eqnarray*}
T=\begin{ytableau}
1&{\Large{\textcir{\normalsize{1}}}}&1&2&*(yellow)2&*(green)4\\2& {\Large{\textcir{\normalsize{2}}}}&3&*(red)3\\34&{\Large{\textcir{\normalsize{4}}}}\\
\end{ytableau} \longrightarrow
\begin{ytableau}
1&{\Large{\textcir{\normalsize{1}}}}1&2&2\\2& {\Large{\textcir{\normalsize{2}}}}3 \textcolor{red}{3}&4\\34&{\Large{\textcir{\normalsize{4}}}}\\
\end{ytableau}=T'
\end{eqnarray*}
Note that $T$ is the last tableau in example \ref{outex} and $T'$ is the first tableau in example \ref{outex}.
\end{ex}
\end{proof}

\begin{proof}[Conclusion of proof of Proposition \ref{bjs}]

We prove that there exists a bijection $\psi_k: MT_k(\lambda) \rightarrow \bigcup\limits_{\nu \supseteq \lambda} MT_{k+1}(\nu)$. If $T \in MT_k(\lambda)$ we define $\psi_k(T)$ simply by applying $\textit{out}$ until column $k$ only contains single entries. This is an element of $\bigcup\limits_{\nu \supseteq \lambda} MT_{k+1}(\nu)$ because of the properties (1), (2), and (3) of $\textit{out}$. If $T \in \bigcup\limits_{\nu \supseteq \lambda} MT_{k+1}(\nu)$ we define $\psi_k^{-1}(T)$ by successively applying $\textit{in}_b$ to the rightmost box $b$ that lies outside of the shape of $\lambda$, until the result has shape $\lambda$. This is an element of $MT_k(\lambda)$ because of the property (2) of $\textit{in}_b$. If $T \in MT_k(\lambda)$ then $\psi_k^{-1}(\psi_k(T))=T$ because of property (3) of $\textit{out}$ and property (1) of how $\textit{out}$ and $\textit{in}_b$ are related. If $T \in \bigcup\limits_{\nu \supseteq \lambda} MT_{k+1}(\nu)$ then $\psi_k(\psi_k^{-1}(T))=T$ by property (3) of $\textit{in}_b$ and property (2) of how $\textit{out}$ and $\textit{in}_b$ are related.

\end{proof}

\begin{thm}\label{mon}
Let $\mu$ be a partition with $n$ parts (some may equal $0$) and longest part equal to $\ell$ and set $\mathbf{t}=(t_1,\ldots , t_{\ell})$, $\mathbf{x}=(x_1,\ldots,x_n)$ then
\begin{eqnarray*}
\mathfrak{J}_{\mu}(x_1,\ldots,x_n,t_1,\ldots,t_{\ell})=
\sum_{P \in MT(\mu)}
\mathbf{t}^{cw(P)}\mathbf{x}^{wt(P)}
\end{eqnarray*}
\end{thm}

\begin{proof} We have:\\

$ \mathfrak{J}_{\mu}(x_1,\ldots,x_n,t_1,\ldots,t_{\ell}) $
\begin{align*}
&=&
\sum_{\lambda \supseteq \mu} \sum_{R \in RT(\lambda/\mu)}
\mathbf{t}^{wt(R)} s_{\lambda}(\mathbf{x}) && \text{Theorem \ref{spex}}& \\
&=&
\sum_{\lambda \supseteq \mu} \sum_{R \in RT(\lambda/\mu)} \sum_{Q \in SSYT(\lambda)}
\mathbf{t}^{wt(R)}\mathbf{x}^{wt(Q)}&& \text{Definition of $s_{\lambda}$} &\\
&=&
\sum_{P \in MT(\mu)}
\mathbf{t}^{cw(P)}\mathbf{x}^{wt(P)}&& \text{Proposition \ref{bjs}}& \\
\end{align*}

The definition of $s_{\lambda}$ mentioned here is the combinatorial one (see 7.10 of \cite{STA99}).

\end{proof}

\begin{remark}
There is a natural crystal structure on the set of semistandard Young tableaux ~\cite{BS17}. Moreover, it is not difficult to see that the bijection $\Psi$ has the property that whenever $\Psi(P)=(Q,R)$ then $P \in \overline{MT}(\mu)$ if and only if $Q$ is highest weight. Thus $\Psi^{-1}$ induces a natural crystal structure on $MT(\mu)$ where the highest weight elements are precisely those that lie in $\overline{MT}(\mu)$. This crystal structure is interpreted algebraically by comparing Corollary \ref{hwns} (where the sum is over highest weight elements) with Theorem \ref{mon} (where the sum is over all elements). This crystal structure coincides with that given in ~\cite{HS20}.
\end{remark}


\section{$\mathfrak{P}_{\mu}(\mathbf{x},\mathbf{t})$ and shifted multiset tableaux}

\subsection{Algebraic Definition of $\mathfrak{P}_{\mu}(\mathbf{x},\mathbf{t})$}

For a strict partition $\mu$ of $m$ nonzero parts, we define the weak symmetric $P$-Grothendieck polynomial in $n \geq m$ variables by:\\

$V*\mathfrak{P}_{\mu}(x_1,\ldots,x_n)=$
\begin{eqnarray*}
&\sum\limits_{\sigma \in S_n/S_{n-m}}  \mathbf{sgn}  (\sigma) \left(\prod\limits_{i} \left(\frac{x_{\sigma_i}}{1-x_{\sigma_i}}\right)^{\mu_i}\right) \left(\prod\limits_{i<j, i\leq m} x_{\sigma_i}+x_{\sigma_j}\right) \left(\prod\limits_{m<i<j} x_{\sigma_i}-x_{\sigma_j}\right)&
\end{eqnarray*}
where:
\begin{eqnarray*}
S_n/S_{n-m}=\{(\sigma_1,\ldots,\sigma_n) \in S_n:  \sigma_i<\sigma_{i+1}, \,\, \forall i: m < i <n \}
\end{eqnarray*}
We define a slight generalization of this polynomial. Suppose $\mu$ has longest part $\ell=\mu_1$. Let the weak symmetric $P$-Grothendieck polynomial in $n + \ell$ variables be defined by:\\

$V*\mathfrak{P}_{\mu}(x_1,\ldots,x_n,t_1,\ldots,t_{\ell})=$
\begin{eqnarray*}
&\sum\limits_{\sigma \in S_n/S_{n-m}}  \mathbf{sgn}  (\sigma) \left(\prod\limits_{i} \left(\frac{x_{\sigma_i}}{1-t_{\ell}x_{\sigma_i}}\right)\cdots \left(\frac{x_{\sigma_i}}{1-t_{\ell-\mu_i+1}x_{\sigma_i}}\right) \right) \left(\prod\limits_{i<j, i\leq m} x_{\sigma_i}+x_{\sigma_j}\right) \left(\prod\limits_{m<i<j} x_{\sigma_i}-x_{\sigma_j}\right)&
\end{eqnarray*}
Clearly $\mathfrak{P}_{\mu}(x_1,\ldots,x_n,1,\ldots,1)=\mathfrak{P}_{\mu}(x_1,\ldots,x_n)$ whereas $\mathfrak{P}_{\mu}(x_1,\ldots,x_n,0,\ldots,0)=P_{\mu}(x_1,\ldots,x_n)$, the $P$-Schur polynomial (see for instance III.8 of \cite{MA79}). Note that the coefficient of $ t_1^{T_1} \cdots t_{\ell}^{T_{\ell}}$ in $V*\mathfrak{P}_{\mu}(x_1,\ldots,x_n,t_1,\ldots,t_{\ell})$ is given by:

\begin{eqnarray*}
\sum_{\sigma \in S_n/S_{n-m}}  \mathbf{sgn}  (\sigma)h_{T_{\ell}}(x_{\sigma_1},\ldots, x_{\sigma_{c_{\ell}}}) \cdots h_{T_1}(x_{\sigma_1},\ldots, x_{\sigma_{c_1}}) \\ *x_{\sigma_1}^{\mu_1} \cdots x_{\sigma_m}^{\mu_m} \left(\prod\limits_{i<j, i\leq m} x_{\sigma_i}+x_{\sigma_j}\right) \left(\prod\limits_{m<i<j} x_{\sigma_i}-x_{\sigma_j}\right)
\end{eqnarray*}
where $(c_{\ell},\ldots,c_1)=\mu'$. We can create each permutation in $S_n/S_{n-m}$ by first selecting $m$ variables and then permuting them.  Now letting:
\begin{eqnarray*}
S_n/(S_{n-m} \times S_m)=\{(\sigma_1,\ldots,\sigma_n) \in S_n:  \sigma_i<\sigma_{i+1}, \,\, \forall i \neq m\}
\end{eqnarray*}
this yields:
\begin{eqnarray*}
\sum\limits_{\tau \in S_n/(S_{n-m} \times S_m)}  \mathbf{sgn}  (\tau) \sum\limits_{\sigma \in S_m(\tau_1,\ldots,\tau_m)}  \mathbf{sgn}  (\sigma)h_{T_{\ell}}(x_{\sigma(\tau_1)},\ldots, x_{\sigma(\tau_{c_{\ell}})}) \cdots h_{T_1}(x_{\sigma(\tau_1)},\ldots, x_{\sigma(\tau_{c_{1}})}) \\ *x_{\sigma(\tau_1)}^{\mu_1} \cdots x_{\sigma(\tau_m)}^{\mu_m} \left(\prod\limits_{i<j \leq m} x_{\sigma(\tau_i)}+x_{\sigma(\tau_j)}\right) \left(\prod\limits_{i\leq m, j>m} x_{\sigma(\tau_i)}+x_{\tau_j}\right) \left(\prod\limits_{m<i<j}x_{\tau_i}-x_{\tau_j}\right)
\end{eqnarray*}
The last three products are constant over the choice of $\sigma$ so we may apply Lemma \ref{hmult} since again $n \geq c_{\ell} \geq \cdots \geq c_1$. We are left with:
\begin{eqnarray*}
\sum\limits_{\tau \in S_n/(S_{n-m} \times S_m)}  \mathbf{sgn}  (\tau) \sum_{\lambda=\lambda^{\ell} \supseteq \cdots \supseteq \lambda^1 \supseteq \lambda^0=\mu} \,\,\,
\sum\limits_{\sigma \in S_m(\tau_1,\ldots,\tau_m)}  \mathbf{sgn}  (\sigma) x_{\sigma(\tau_1)}^{\lambda_1} \cdots x_{\sigma(\tau_m)}^{\lambda_m} \\ * \left(\prod\limits_{i<j \leq m} x_{\sigma(\tau_i)}+x_{\sigma(\tau_j)}\right) \left(\prod\limits_{i\leq m, j>m} x_{\sigma(\tau_i)}+x_{\tau_j}\right) \left(\prod\limits_{m<i<j}x_{\tau_i}-x_{\tau_j}\right)
\end{eqnarray*}
where the middle sum is over all good $\mathbf{T}$-extensions. Reverting to a sum over a single set of permutations this becomes:
\begin{eqnarray*}
\sum_{\lambda=\lambda^{\ell} \supseteq \cdots \supseteq \lambda^1 \supseteq \lambda^0=\mu} \,\,\,
\sum_{\sigma \in S_n/S_{n-m}}  \mathbf{sgn}  (\sigma) x_{\sigma_1}^{\lambda_1} \cdots x_{\sigma_m}^{\lambda_m} \left(\prod\limits_{i<j, i\leq m} x_{\sigma_i}+x_{\sigma_j}\right) \left(\prod\limits_{m<i<j} x_{\sigma_i}-x_{\sigma_j}\right)
\end{eqnarray*}

Fix some $\lambda \supseteq \mu$ and let $\rho=\lambda-\delta$ where $\delta=(m-1,\ldots,0)$ and consider the set of semistandard Young tableaux of shape $\rho/(\mu-\delta)$ and weight $T_1,\ldots,T_{\ell}$ such that every entry $i$ occurs on or above row $c_i$.  We claim that each good $\mathbf{T}$-extension $\lambda=\lambda^{\ell} \supseteq \cdots \supseteq \lambda^1 \supseteq \lambda^0=\mu$ encodes such a tableau in a bijective fashion.  Indeed, the bijection is given by first considering the tableau of shape $\lambda/\mu$ where each strip $\lambda^i/\lambda^{i-1}$ is filled with $i$s for each $i \in [1, \ell]$ and then,  moving all entries in row $j$ to the left by $m-j$ positions  for each $j \in [1,m]$ (and eliminating excess boxes).

 Since the inner sum in the equation above is precisely the classical definition of $V*P_{\lambda}(x_1,\ldots,x_n)$ where $P_{\lambda}(x_1,\ldots,x_n)$ is the $P$-Schur function \cite{MA79}, it now follows that the expression above is the same as:

\begin{eqnarray*}
V*\sum_{\rho \supseteq (\mu-\delta)}
(N^{\mathbf{T}}_{\rho/(\mu-\delta)}) P_{\rho+\delta}(x_1,\ldots,x_n)
\end{eqnarray*}
where $N^{\mathbf{T}}_{\rho/(\mu-\delta)}$ is the number of semistandard Young tableaux of shape $\rho/(\mu-\delta)$ and weight $T_1,\ldots,T_{\ell}$ such that every entry $i$ occurs on or above row $c_i$.

\begin{dfn}
Let $\mu$ be a partition with $m$ distinct, nonzero parts and conjugate $\mu'=(c_{\ell},\ldots,c_1)$ and set $\delta=(m-1,\ldots,0)$. If $\lambda \supseteq \mu$ is a partition of $m$ distinct parts then a \emph{\textbf{shifted restricted tableau}} of shape $(\lambda-\delta)/(\mu-\delta)$ is a semistandard Young tableau of this shape using entries in the alphabet $\{1,\ldots,\ell\}$ such that each entry $i$ occurs on or above row $c_i$. We denote the set of all such tableaux by $SRT(\lambda/\mu)$. If $R \in SRT(\lambda/\mu)$ then the weight of $R$, denoted $wt(R)$ is the vector $(w_1,\ldots,w_{\ell})$ where $w_i$ is the number of $i$s that appear in $R$.
\end{dfn}

\begin{ex}\label{resexs}
Let $\lambda=(10,8,6,4)$ and $\mu =(7,5,4,2)$ so that $c_7=4$, $c_6=4$, $c_5=3$, $c_4=3$, $c_3=2$, $c_2=1$, $c_1=1$.  (Since the diagram below is shown as a shifted diagram the $c_i$ are computed by counting the length of the top left to bottom right diagonals of $\mu$.)
\begin{eqnarray*}
\ytableausetup{notabloids}
\begin{ytableau}
\cdot & \cdot & \cdot & *(green) \cdot & *(green) \cdot &*(green) \cdot &*(green) \cdot & *(green) 2 & *(green) 3 & *(green) 5\\
\none&\cdot & \cdot & *(yellow) \cdot & *(yellow) \cdot &*(yellow) \cdot & *(yellow) 3 & *(yellow) 3 & *(yellow) 6\\
\none& \none & \cdot& *(orange) \cdot & *(orange) \cdot &*(orange) \cdot & *(orange) 4 & *(orange) 7 \\
\none & \none & \none & *(red) \cdot &*(red) \cdot & *(red) 6 & *(red) 7 \\
\end{ytableau}
\end{eqnarray*}
Since all 1s and 2s lie in the green all 3s lie in the green or yellow, all 4s and all 5s lie in the orange, yellow, or green, and all 6s and 7s lie in the red, orange, yellow, or green, this is an element of $SRT(\lambda/\mu)$. It has weight $(0,1,3,1,1,2,2)$.
\end{ex}

The statement before the definition now becomes:

\begin{thm}\label{Pex}
Let $\mu$ be a partition with $m \leq n$ distinct nonzero parts and longest part equal to $\ell$ and set $\mathbf{t}=(t_1,\ldots , t_{\ell})$, $\mathbf{x}=(x_1,\ldots,x_n)$ then

\begin{eqnarray*}
\mathfrak{P}_{\mu}(x_1,\ldots,x_n,t_1,\ldots,t_{\ell})=
\sum_{\lambda \supseteq (\mu-\delta)} \sum_{R \in SRT((\lambda+\delta)/\mu)}
\mathbf{t}^{wt(R)} P_{\lambda+\delta}(\mathbf{x})
\end{eqnarray*}
\end{thm}

\begin{remark}
Note that $RT(\lambda/\mu)$ is \emph{not} the 
 $SRT((\lambda+\delta)/(\mu+\delta))$ since in the first case we use the constants $(c_{\ell},\ldots,c_1)=\mu'$ and the alphabet $\{1,\ldots,\ell\}$ and in the second we would use the constants $(d_{\ell+m-1},\ldots,d_1)=(\mu+\delta)'$ and the alphabet $\{1,\ldots,\ell+m-1\}$.
\end{remark}

\subsection{Shifted shape multiset tableaux}

In this section we will use the following ordered entries to fill tableaux: $S'=\{1'< 1< 2' < 2 < 3' < \cdots\}$. We use the following notation. Let $a,z \in S'$

\begin{itemize}
\item{ $a <_u z$ means $a<z$ or else $a=z$ and they are unprimed.}
\item{$a <_p z$ means $a<z$ or else $a=z$ and they are primed.}
\item{ $a >_u z$ means $a>z$ or else $a=z$ and they are unprimed.}
\item{$a >_p z$ means $a>z$ or else $a=z$ and they are primed.}
\end{itemize}
\begin{dfn}
Given a partition with distinct parts, $\mu=(\mu_1,\ldots,\mu_{m})$, a \textbf{\emph{signed shifted multiset tableau}} of shape $\mu$, or element of $SMT(\mu)$, is an arrangement of boxes with $\mu_i$ adjacent boxes in row $i$ for each $i$ and where the rows are situated such that the leftmost box of row $i$ lies one column to the left of the leftmost box of row $i+1$, along with a filling of said boxes with the following properties.
\begin{enumerate}
\item{Each box contains a nonempty multiset of the numbers $\{1',1,2',2,3',\ldots\}$ such that the multiplicity of each primed number is $0$ or $1$.}
\item{Suppose entry $z$ lies in a box directly to the right of box $b$. Then for \emph{all} $a \in b$ we have $a <_u z$.}
\item{Suppose entry $z$ lies in a box directly below box $b$. Then for \emph{some} $a \in b$ we have $a <_p z$.}
\end{enumerate}
If, in addition the smallest entry in each row is not primed we call such a tableau simply a \textbf{\emph{shifted multiset tableau}} of shape $\mu$ or an element of $SMT^0(\mu)$.\footnote{Compare to the definitions of \emph{weak set-valued shifted tableaux} in ~\cite{HKPWZZ17} and \emph{set-valued shifted tableaux} in ~\cite{IN13}.}
\end{dfn}
The weight of a (signed) shifted multiset tableau is the vector $(w_1,w_2, \ldots)$ where $w_i$ is the total number of $i$s or $i'$s appearing in the tableau. We label the top left to bottom right diagonals from left to right by $\{\ell, \ell-1,\ldots,2,1\}$ where $\ell=\mu_1$. By box $d_{ij}$ we refer to the box that is in the $i^{th}$ row (from top to bottom) of diagonal $j$. Define the diagonal weight of a shifted multiset tableau, $dw$, to be the vector $(T_1,\ldots,T_{\ell})$ where $T_j$ is the difference between the number of entries in diagonal $j$ and the number of boxes in diagonal $j$. Let, $|d_{ij}|$ mean the total number of entries in box $d_{ij}$ and $|d_{ij}(x)|$ refer, more specifically, to the number of entries in box $d_{ij}$ in tableau $x$. The convention is $|d_{ij}|=0$ if $d_{ij}$ describes a position not in the tableau.

\begin{ex}
Let $\mu=(5,4,2)$. Then
\begin{eqnarray*}
\begin{ytableau}
1&1113&3&4'45&7'7\\ \none &22&4'4&5'6'&7'\\ \none & \none & 45'&55\\
\end{ytableau}
\end{eqnarray*}
is an element $P \in SMT(\mu)$ with $wt(P)=(4,2,2,5,5,1,3)$ and $dw(P)=(1,2,1,5,2)$.
\end{ex}

\begin{dfn}
An element of ${SMT}(\mu)$ with diagonal weight $(0,\ldots,0)$ is called a \textbf{\emph{signed shifted semistandard tableau}} of shape $\mu$, or element of ${SST}(\mu)$. An element of ${SMT^0}(\mu)$ with diagonal weight $(0,\ldots,0)$ is called a \textbf{\emph{shifted semistandard tableau}} of shape $\mu$, or element of ${SST^0}(\mu)$.
\end{dfn}

\begin{remark}
Note that ${SST^0}(\mu)$, which is the subset of ${SST}(\mu)$ with no primes in the leftmost diagonal, agrees with the usual definition of shifted semistandard tableau (e.g., ~\cite{Ser09}) and is therefore the generating set for the $P$-Schur function $P_{\mu}$. Moreover, if $m$ is the number of parts of $\mu$, it is not difficult to see that number of elements with a given weight and column weight in ${SST}(\mu)$ is equal to $2^m$ times the number of elements in ${SST^0}(\mu)$ with that weight and column weight. The relationship of  ${SMT}(\mu)$ to ${SMT^0}(\mu)$ is the same.
\end{remark}

\begin{dfn}
A \textbf{\emph{maximal shifted multiset tableau}} of shape $\mu$, or element of $\overline{SMT}(\mu)$ is an element of ${SMT}(\mu)$ with the following properties:

\begin{enumerate}
\item{Each box $d_{ij}$ may only contain $i$s.}
\item{For each $i \geq 1$ and $ k \geq 0$ we have $\sum\limits_{1 \leq j \leq k} |d_{(i+1)j}|-|d_{i(j-1)}| \leq 0$}
\end{enumerate}

\begin{ex} \label{maxexs}
Let $\mu=(4,3,3,1)$. Then
\begin{eqnarray*}
\begin{ytableau}
1&1&11&1&11&11&1\\ \none &2&22&2&2&222\\ \none & \none & 33&3&3&33\\ \none & \none & \none & 44&44
\end{ytableau}
\end{eqnarray*}
is an element $P \in \overline{MT}(\mu)$ with $wt(P)=(7,6,5,4)$ and $cw(P)=(1,3,4,2)$.
\end{ex}

\end{dfn}

\begin{prop}
There is a bijection from the subset of $\overline{SMT}(\mu)$ with weight $\lambda$ and diagonal weight $\mathbf{T}$ to the subset of $SRT(\lambda/\mu)$ with weight $\mathbf{T}$.
\end{prop}

\begin{proof}
Let $X$ be the subset of elements of $SMT(\mu)$ with weight $\lambda$ and diagonal weight $\mathbf{T}$ that satisfy property (1) above. Let $Y$ be the set of weakly increasing along rows fillings (one entry per box) of shape $(\lambda-\delta)/(\mu-\delta)$ and weight $\mathbf{T}$ such that every entry $i$ occurs on or above row $c_i$ (equivalently, row $i$ only contains entries greater than $\ell-\mu_i$). The map $x \rightarrow y$ where $y$ is defined by the property that for each $(i,j)$, row $i$ of $y$ contains exactly $|d_{ij}(x)|-1$ copies of $j$ is a bijection from $X$ to $Y$. Moreover if $x \rightarrow y$ then $x$ satisfies property (2) above if and only if the columns of $y$ are strictly decreasing down rows: Indeed, if there is some $i$ and some $k$ such that $\sum\limits_{1 \leq j \leq k} |d_{(i+1)j}|-|d_{i(j-1)}| >0$ then for the minimal such $k$, the rightmost $k$ appearing in row $i+1$ will be in diagonal $k+1-\sum\limits_{1 \leq j \leq k} |d_{(i+1)j}|$ whereas the rightmost $k-1$ (or rightmost instance of the greatest number less than $k-1$ if this row contains no $k-1$) appearing in row $i$ of $y$ will lie in column $k-\sum\limits_{1 \leq j \leq k} |d_{i(j-1)}|$, that is, strictly to the left of the aforementioned $k$. Hence said $k$ must lie below a number greater than or equal to $k$.

On the other hand, if in some diagonal, say $p$, row $i+1$ of $y$ contains a $k$ and diagonal $p-1$ row $i$ contains a $k'$  with $k'\geq k$ then for the minimal such $p$, we have $k+1-\sum\limits_{1 \leq j \leq k} |d_{(i+1)j}| \leq p$  and $k-\sum\limits_{1 \leq j \leq k} |d_{i(j-1)}| \geq p$ so that we have $\sum\limits_{1 \leq j \leq k} |d_{(i+1)j}|-|d_{i(j-1)}| >0$. Since the elements of $Y$ that are strictly decreasing down columns are exactly the elements of $SRT(\lambda/\mu)$ with weight $\mathbf{T}$, the map restricted to the elements of $X$ that satisfy property (2) gives the desired bijection.
\end{proof}

\begin{ex}
The tableaux of examples \ref{resexs} and \ref{maxexs} correspond under this bijection.
\end{ex}

\begin{cor}\label{hws}
Let $\mu$ be a partition with $m \leq n$ distinct nonzero parts and longest part equal to $\ell$ and set $\mathbf{t}=(t_1,\ldots , t_{\ell})$, $\mathbf{x}=(x_1,\ldots,x_n)$ then
\begin{eqnarray*}
\mathfrak{P}_{\mu}(x_1,\ldots,x_n,t_1,\ldots,t_{\ell})=
\sum_{Q \in \overline{SMT}(\mu)}
\mathbf{t}^{dw(Q)}P_{wt(Q)}(\mathbf{x})
\end{eqnarray*}
\end{cor}

\subsection{Combinatorial Definition of $\mathfrak{P}_{\mu}(\mathbf{x},\mathbf{t})$}

In this section we will give an equivalent combinatorial definition of $\mathfrak{P}_{\mu}$. We will need a certain column insertion algorithm. In the below, we describe how to \textbf{insert} and \textbf{reverse insert} an element into a column.

Let $K$ be a valid column (each box of $K$ contains exactly one element from $S'$ and whenever $a$ lies above $z$ in $K$ we have $a <_p z$). Now let $a \in S'$. We \textbf{insert} $a$ into $K$, denoted $a \hookrightarrow K$ as follows:
Let $\hat{a}$ denote the uppermost entry in $K$ such that $a<_u \hat{a}$. If $\hat{a}$ exists, replace $\hat{a}$ with $a$ and bump out $\hat{a}$. Otherwise, append $a$ to the bottom of $K$. The result is recorded as the pair $(K',\hat{a})$ if the second of this pair exists and just $K'$ otherwise.
On the other hand if $z \in S'$ is any element such that $z >_u a$ for some $a \in K$ then we define \textbf{reverse insertion} of $z$ into $K$ as follows: Let $\hat{z}$ denote the bottommost entry in $K$ such that $z>_u \hat{z}$. Replace $\hat{z}$ with $z$ and bump out $\hat{z}$. The result is recorded as the pair $(\hat{z},K')$.

Notice the basic properties:

\begin{enumerate}
\item{If $a \hookrightarrow K=K'$ then $K'$ is a valid column. }
\item{if $a \hookrightarrow K= (K',\hat{a})$ then $K'$ is a valid column.}
\item{If $K \hookleftarrow z=(\hat{z},K')$ then $K'$ is a valid column.}
\item{If $a <_u z$ then either
\begin{itemize}
\item{$z \hookrightarrow K=K'$ and $a \hookrightarrow K' =({K}^{\prime\prime},\hat{a})$.}
\item{$z \hookrightarrow K=(K',\hat{z})$ and $a \hookrightarrow K' =({K}^{\prime\prime},\hat{a})$ where $\hat{a} <_u \hat{z}$.}
\end{itemize}}
\item{If $a <_u z$ and
$K \hookleftarrow a=(\hat{a},K')$ and $K' \hookleftarrow z=(\hat{z},{K}^{\prime\prime})$ then $\hat{a} <_u \hat{z}$.}
\end{enumerate}

Now, fix a partition $\mu$ with $m$ distinct nontrivial parts and with conjugate $\mu'=(c_{\ell},\ldots,c_1)$. We will refer to both columns and diagonals. Both are labeled in decreasing order from left to right starting on $\ell$.

\begin{prop} \label{bijs}
There is a bijection $SMT(\mu) \rightarrow \bigcup\limits_{\lambda \supseteq \mu} SST(\lambda) \times SRT(\lambda/\mu)$, such that if $P \rightarrow (Q,R)$ then:
\begin{enumerate}
\item{$wt(P)=wt(Q)$.}
\item{$dw(P)=wt(R)$.}
\end{enumerate}

\end{prop}
\begin{proof}[Proof preliminaries]\renewcommand{\qedsymbol}{}
We define the set $SMT_{k}(\lambda)$ to be the subset of elements of $SMT(\lambda)$ that have only single entries in diagonals $k-1,\ldots,1,0,-1,\ldots$. Define the set $SRT_{k}(\lambda/\mu)$ to be the subset of elements of $SRT(\lambda/\mu)$ that have only entries from $\{1,2,\ldots,k-1\}$. Given a pair $(Q,R) \in SMT_k(\lambda) \times SRT_k(\lambda/\mu)$ define the weight and diagonal weight of this pair as $wt(Q,R)=wt(Q)$ and $dw(Q,R)=dw(Q)+wt(R)$. To  prove the proposition it suffices to find a weight and diagonal weight preserving bijection for each $k$  from $\bigcup\limits_{\lambda \supseteq \mu} SMT_k(\lambda) \times SRT_k(\lambda/\mu)$ to $\bigcup\limits_{\lambda \supseteq \mu} SMT_{k+1}(\lambda) \times SRT_{k+1}(\lambda/\mu)$ (and then compose $\Phi(P)=\Phi_{\ell} \circ \cdots \circ \Phi_1 (P,Q_0)$ where $Q_0$ is the empty tableau of shape $\mu/\mu$). To do the former, it is enough to find a weight preserving bijection $\phi_k: SMT_k(\lambda) \rightarrow \bigcup\limits_{\nu \supseteq \lambda} SMT_{k+1}(\nu)$ where the union is over all $\nu$ such $\nu/\lambda$ is a horizontal strip with no box below row $c_{k}$ (recall that the value of $c_k$ is defined by the shape $\mu$ although it is easily shown that if  we label the leftmost column of $\lambda$ by $\ell=\mu_1$ then diagonal $k$ of $\mu$ has the same length as diagonal $k$ of $\lambda$).   We then set $\Phi_k(Q,R)=(\phi_k(Q),R')$ where $R'$ is obtained by appending boxes to $R$ until its outer  shape is the same as $\phi_k(Q)$  and filling each appended box with the entry $k$.

Before we can define $\phi_k$ we need to introduce the following map: Let $T \in SMT_{k}(\lambda)$. Define $\textbf{out}(T)$
as follows: First, in each box of diagonal $k$ circle (one of) the minimum entry(s) from that box. Now find (one of) the largest noncircled entry(s) in diagonal $k$ and remove it and \textbf{insert} it into the undercolumn to the right of the column from which it was removed (where the \emph{undercolumn} denotes the part of the column that lies below a circled entry, or, if there is no circled entry in the column, the entire column). After this, each time an element is bumped, \textbf{insert} it into the next undercolumn to the right until some entry is eventually appended to an undercolumn. Note the following properties of $\textbf{out}$.

\begin{enumerate}
\item{The path of positions where an element is bumped/appended moves weakly up as we move to the right.}
\item{Properties (1), (2), and (3) in the definition of shifted multiset tableaux are preserved under $\textbf{out}$.}
\item{If $\textbf{out}(T)$ and $\textbf{out}(\textbf{out}(T))$ are both defined then the box that $\textbf{out}$ appends to $\textbf{out}(T)$ lies strictly to the right of the box that $\textbf{out}$ appends to $T$.}
\end{enumerate}

\ytableausetup{boxsize=24pt}
\begin{ex} \label{outexs}
Suppose that diagonal $k$ is the second diagonal from the left.  Each $\longrightarrow$ represents an application of $\textbf{out}$.
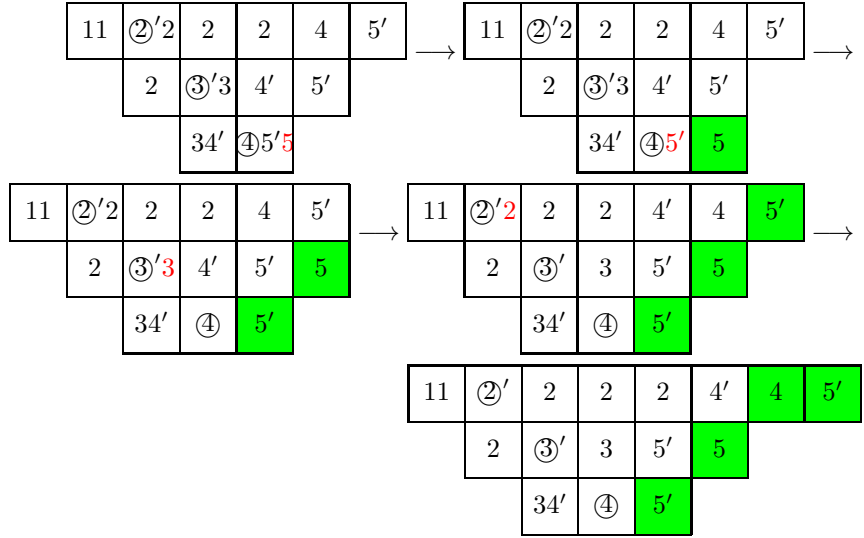
\begin{eqnarray*}
\begin{ytableau}
11&{\Large{\textcir{\normalsize{2}}}}'2&2&2&4&5'\\ \none & 2& {\Large{\textcir{\normalsize{3}}}}'3&4'&5'\\ \none & \none & 34'&{\Large{\textcir{\normalsize{4}}}}5'\textcolor{red}{5}\\
\end{ytableau}\longrightarrow
\begin{ytableau}
11&{\Large{\textcir{\normalsize{2}}}}'2&2&2&4&5'\\ \none & 2& {\Large{\textcir{\normalsize{3}}}}'3&4'&5'\\ \none & \none & 34'&{\Large{\textcir{\normalsize{4}}}}\textcolor{red}{5'}&*(green)5\\
\end{ytableau}\longrightarrow \\
\begin{ytableau}
11&{\Large{\textcir{\normalsize{2}}}}'2&2&2&4&5'\\ \none & 2& {\Large{\textcir{\normalsize{3}}}}'\textcolor{red}{3}&4'&5'&*(green)5\\ \none & \none & 34'&{\Large{
{\normalsize{4}}}}&*(green)5'\\
\end{ytableau}\longrightarrow
\begin{ytableau}
11&{\Large{\textcir{\normalsize{2}}}}'\textcolor{red}{2}&2&2&4'&4&*(green)5'\\ \none & 2& {\Large{\textcir{\normalsize{3}}}}'&3&5'&*(green)5\\ \none & \none & 34'&{\Large{\textcir{\normalsize{4}}}}&*(green)5'\\
\end{ytableau}\longrightarrow \\
\begin{ytableau}
11&{\Large{\textcir{\normalsize{2}}}}'&2&2&2&4'&*(green)4&*(green)5'\\ \none & 2& {\Large{\textcir{\normalsize{3}}}}'&3&5'&*(green)5\\ \none & \none & 34'&{\Large{\textcir{\normalsize{4}}}}&*(green)5'\\
\end{ytableau}
\end{eqnarray*}
Uncircled numbers being removed are shown in red, and the boxes being added appear in green.
\end{ex}

We will also need a map called $\textbf{in}_b$. Let $T \in SMT_{k}(\nu)$ for some $\nu$ such that $\nu/\lambda$ is a horizontal strip with no box below row $c_{k}$ and suppose $b$ is some corner box of $T$ that lies on or above row $c_k$. Define $\textbf{in}_b(T)$ as follows: First, in each box of diagonal $k$ circle (one of) the minimum entry(s) from that box. Now remove the entry from box $b$. If this entry is less than the circled entry in the column to the left, or both are equal and primed, or there is no such circled element, \textbf{reverse insert} it into the undercolumn of the column to the left. After this, each time an element is bumped that is less than the circled entry in the column to its left or equal to it and primed, \textbf{reverse insert} it into the undercolumn of the column to the left. When an element is bumped that is greater than the circled entry in the column to its left or equal to it and unprimed, add it to the box containing this circled element. Note the following properties of $\textbf{in}_b$.

\begin{enumerate}
\item{The path of positions where an element is bumped/added moves weakly down as we move to the left.}
\item{Properties (2), and (3) in the definition of shifted multiset tableaux are preserved under $\textbf{in}_b$. Property (1) is satisfied unless $\textbf{in}_b$ adds a primed entry to a box already containing a the same noncircled primed entry.}
\item{If $b'$ lies to the left of $b$ and if $\textbf{in}_b(T)$ and $\textbf{in}_{b'}(\textbf{in}_b(T))$ are both defined then the element that $\textbf{in}_{b'}$ adds to diagonal $k$ of $\textbf{in}_b(T)$ is greater than, or equal to and unprimed, the element $\textbf{in}_b$ adds to diagonal $k$ of $T$. }
\end{enumerate}

Moreover, $\textbf{out}$ and $\textbf{in}_b$ are related as follows:

\begin{enumerate}
\item{If $\textbf{out}$ appends box $b$ when applied to $T$, then $\textbf{in}_b(\textbf{out}(T))=T$.}
\item{If the element that $\textbf{in}_b$ adds to diagonal $k$ when applied to $T$ is the largest, or tied for the largest and unprimed, uncircled element on diagonal $k$ then $\textbf{in}_b(T)$ satisfies property (1) in the definition of shifted multiset tableaux (and hence is a shifted multiset tableau), and $\textbf{out}(\textbf{in}_b(T))=T$ .}
\end{enumerate}

\begin{ex} \label{uuu}
Set $k=2$. Then $\textbf{in}_{red}(\textbf{in}_{orange}(\textbf{in}_{yellow}(\textbf{in}_{green}(T))))=T'$ where $T$ is the tableau below on the left and $T'$ is the tableau below on the right.
\begin{eqnarray*}
\begin{ytableau}
11&{\Large{\textcir{\normalsize{2}}}}'&2&2&2&4'&*(yellow)4&*(green)5'\\ \none & 2& {\Large{\textcir{\normalsize{3}}}}'&3&5'&*(orange)5\\ \none & \none & 34'&{\Large{\textcir{\normalsize{4}}}}&*(red)5'\\
\end{ytableau} \longrightarrow
\begin{ytableau}
11&{\Large{\textcir{\normalsize{2}}}}'2&2&2&4&5'\\ \none & 2& {\Large{\textcir{\normalsize{3}}}}'3&4'&5'\\ \none & \none & 34'&{\Large{\textcir{\normalsize{4}}}}5'\textcolor{red}{5}\\
\end{ytableau}
\end{eqnarray*}
Note that $T$ is the last tableau of example \ref{outexs} and $T'$ is the first tableau of \ref{outexs}.
\end{ex}
\end{proof}

\begin{proof}[Conclusion of proof of Proposition \ref{bijs}]
We define $\phi_k$ simply by applying $\textbf{out}$ until diagonal $k$ only contains single entries.
\begin{enumerate}
\item{$\phi_k$ is well defined. For any tableau $T$ \emph{denote the shape of $T$ by $T^s$}.  If $T \in SMT_k(\lambda)$ then Property (3) of $\textbf{out}$ implies $\phi_k(T)^s/T^s$ is a horizontal strip and Property (1) of $\textbf{out}$ implies all of its boxes lie on or above row $c_k$. On the other hand Property (2) of $\textbf{out}$ implies that $\phi_k(T)$ is a valid shifted multiset tableau, and, by construction $\phi_k(T)$ has only single entries in diagonals $k, k-1, \ldots, 0,-1,\ldots$.}
\item{$ \phi_k$ is injective. Suppose $T \neq T' \in SMT_k(\lambda)$ with $\phi_k(T)=\phi_k(T')$ then by Property (3) of $\textbf{out}$ and construction of $\phi_k$ there is some $\nu$ and some $S \neq S' \in SMT_k(\nu)$ with $\textbf{out}(S)=\textbf{out}(S')$. But then if $b$ is the box that $\textbf{out}$ adds to $S$ or equivalently to $S'$, property (1) of how $\textbf{out}$ and $\textbf{in}_b$ are related says $S=\textbf{in}_{b}(\textbf{out}(S))=\textbf{in}_{b}(\textbf{out}(S'))=S'$.}
\item{$ \phi_k$ is surjective. Let $T \in \bigcup\limits_{\nu \supseteq \lambda} SMT_{k+1}(\nu)$ where the union is over all $\nu$ such $\nu/\lambda$ is a horizontal strip with no box below row $c_{k}$. Let $b_1,\ldots,b_r$ denote the boxes labeled from left to right of $T^s/\lambda$. Set $S=\textbf{in}_{b_1}(\cdots(\textbf{in}_{b_r}(T)\cdots)$. Property (3) of $\textbf{in}_b$ implies that for each $i$ we have that $\textbf{in}_{b_i}$ adds a an element to diagonal $k$ when applied to $\textbf{in}_{b_{i+1}}(\cdots(\textbf{in}_{b_r}(T)\cdots)$ that is the largest, or tied for largest and unprimed, noncircled element in diagonal $k$. This along with property (2) of $\textbf{in}_b$ implies $\textbf{in}_{b_{i}}(\cdots(\textbf{in}_{b_r}(T)\cdots)$ is a valid shifted multiset tableau. Moreover, property (2) of how $\textbf{out}$ and $\textbf{in}_b$ are related says that in this case $\textbf{out}( \textbf{in}_{b_{i}}(\cdots(\textbf{in}_{b_r}(T)\cdots) )=\textbf{in}_{b_{i+1}}(\cdots(\textbf{in}_{b_r}(T)\cdots)$. All together, this implies that $S$ is a valid shifted multiset tableau and that $\phi_k(S)=T$. By construction, $S$ has shape $\lambda$ and has only single entries in diagonals $ k-1, \ldots, 0,-1,\ldots$, i.e., $S \in SMT_k(\lambda)$. }
\end{enumerate}

\end{proof}

\begin{thm}\label{ms}
Let $\mu$ be a partition with $m \leq n$ distinct nonzero parts and longest part equal to $\ell$ and set $\mathbf{t}=(t_1,\ldots , t_{\ell})$, $\mathbf{x}=(x_1,\ldots,x_n)$ then

\begin{eqnarray*}
\mathfrak{P}_{\mu}(x_1,\ldots,x_n,t_1,\ldots,t_{\ell})=
\sum_{P \in SMT^0(\mu)}
\mathbf{t}^{dw(P)}\mathbf{x}^{wt(P)}
\end{eqnarray*}

\end{thm}

\begin{proof}
Let $m$ denote the number of parts of $\mu$.\\

$\mathfrak{P}_{\mu}(x_1,\ldots,x_n,t_1,\ldots,t_{\ell})$
\begin{align*}
&=&
\sum_{\lambda \supseteq (\mu-\delta)} \sum_{R \in SRT((\lambda+\delta)/\mu)}
\mathbf{t}^{wt(R)} P_{\lambda+\delta}(\mathbf{x}) && \text{Theorem \ref{Pex}}& \\
&=&
\sum_{\lambda \supseteq (\mu-\delta)} \sum_{R \in SRT((\lambda+\delta)/\mu)} \sum_{Q \in SST^0(\lambda+\delta)}
\mathbf{t}^{wt(R)}\mathbf{x}^{wt(Q)}&& \text{Def. of $P_{\lambda+\delta}$} &\\
&=&
\sum_{\lambda \supseteq (\mu-\delta)} \sum_{R \in SRT((\lambda+\delta)/\mu)} \sum_{Q \in SST(\lambda+\delta)}
(2^{-m})\mathbf{t}^{wt(R)}\mathbf{x}^{wt(Q)}&& \text{Def. of $SST^0$} &\\
&=&
\sum_{P \in SMT(\mu)}
(2^{-m}) \mathbf{t}^{dw(P)}\mathbf{x}^{wt(P)}&& \text{Prop. \ref{bijs}}& \\
&=&
\sum_{P \in SMT^0(\mu)}
\mathbf{t}^{dw(P)}\mathbf{x}^{wt(P)}&& \text{Def. of $SMT^0$} &
\end{align*}

The definition of $P_{\lambda+\delta}$ mentioned here is the combinatorial one (see the end of section III.8 of \cite{MA79}).
\end{proof}

\begin{ex}
Let us consider $\mathfrak{P}_{2,1}(x_1,x_2,t_1,t_2)$. We will compute the degree 4 part in $\mathbf{x}$ (which is the degree 1 part in $\mathbf{t}$). We have the following tableaux:

\begin{eqnarray*}
& \begin{ytableau}
11&1\\ \none & 2
\end{ytableau}\,\,\,\,\,\,
\begin{ytableau}
1&11\\ \none & 2
\end{ytableau}& \\
&\begin{ytableau}
11&2'\\ \none & 2
\end{ytableau}\,\,\,\,\,\,
\begin{ytableau}
1&1\\ \none & 22
\end{ytableau}\,\,\,\,\,\,
\begin{ytableau}
1&12'\\ \none & 2
\end{ytableau}\,\,\,\,\,\,
\begin{ytableau}
1&12\\ \none & 2
\end{ytableau}& \\
& \begin{ytableau}
1&2'\\ \none & 22
\end{ytableau}\,\,\,\,\,\,
\begin{ytableau}
1&2'2\\ \none & 2
\end{ytableau}&
\end{eqnarray*}
This yields $x_1^3x_2t_1+x_1^3x_2t_2+
2x_1^2x_2^2t_1+2x_1^2x_2^2t_2+
x_1x_2^3t_1+x_1x_2^3t_2$, which can be expressed in terms of $P$-Schur polynomials as $t_1P_{3,1}(x_1,x_2)+t_2P_{3,1}(x_1,x_2)$. Note that this is already different from example 3.3 of ~\cite{HKPWZZ17}.

\end{ex}

\begin{remark}
There exists a queer crystal structure on the set of semistandard shifted tableaux ~\cite{Hir19}. Under this structure, the highest weight elements are precisely those for which every entry on row $i$ is an (unprimed) $i$. Moreover, the bijection $\Phi$ fixes the minimum entry on each row. Thus restricting $\Phi$ gives a bijection from $SMT^0(\mu) \rightarrow \bigcup\limits_{\lambda \supseteq \mu} SST^0(\lambda) \times SRT(\lambda/\mu)$. Moreover, it is not difficult to see that this restriction of $\Phi$ has the property that whenever $\Phi(P)=(Q,R)$ then $P \in \overline{SMT}(\mu)$ if and only if $Q$ is highest weight. Thus $\Phi^{-1}$ induces a queer crystal structure on $SMT^0(\mu)$ where the highest weight elements are precisely those that lie in $\overline{SMT}(\mu)$. This crystal structure is interpreted algebraically by comparing Corollary \ref{hws} (where the sum is over highest weight elements) with Theorem \ref{ms} (where the sum is over all elements).
\end{remark}

\bibliographystyle{alpha}

\newcommand{\etalchar}[1]{$^{#1}$}

\end{document}